\documentclass[11pt]{amsart}
\usepackage{amsmath}
\usepackage{amssymb}
\usepackage{amsthm}
\usepackage{amscd}
\usepackage{enumerate}
\usepackage{verbatim}
\usepackage[all]{xy}
\usepackage{mathrsfs}

\theoremstyle{plain}
\newtheorem{thm}{Theorem}[section]
\newtheorem{prop}[thm]{Proposition}
\newtheorem{lem}[thm]{Lemma}
\newtheorem{cor}[thm]{Corollary}

\theoremstyle{definition}

\newtheorem{rmk}[thm]{Remark}

\newcommand{\rank}{\mathrm{rank}}

\newcommand{\Hom}{\mathrm{Hom}}

\newcommand{\Lie}{\mathrm{Lie}}

\newcommand{\Ker}{\mathrm{Ker}}

\newcommand{\Img}{\mathrm{Im}}
\newcommand{\prjt}{\mathrm{pr}}
\newcommand{\Fil}{\mathrm{Fil}}
\newcommand{\Spf}{\mathrm{Spf}}

\newcommand{\Spec}{\mathrm{Spec}}

\newcommand{\Gal}{\mathrm{Gal}}

\newcommand{\Frac}{\mathrm{Frac}}

\newcommand{\Span}{\mathrm{Span}}

\newcommand{\diag}{\mathrm{diag}}

\newcommand{\Kbar}{\bar{K}}

\newcommand{\okbar}{\mathcal{O}_{\bar{K}}}

\newcommand{\Acrys}{A_{\mathrm{crys}}}

\newcommand{\okey}{\mathcal{O}_K}
\newcommand{\oc}{\mathcal{O}_{\mathbb{C}}}
\newcommand{\oel}{\mathcal{O}_L}

\newcommand{\okeyn}{\mathcal{O}_{K_n}}

\newcommand{\cC}{\mathcal{C}}
\newcommand{\cE}{\mathcal{E}}
\newcommand{\cF}{\mathcal{F}}
\newcommand{\cG}{\mathcal{G}}
\newcommand{\cH}{\mathcal{H}}

\newcommand{\cM}{\mathcal{M}}

\newcommand{\cO}{\mathcal{O}}

\newcommand{\cX}{\mathcal{X}}
\newcommand{\cY}{\mathcal{Y}}

\newcommand{\Ksep}{K^{\mathrm{sep}}}
\newcommand{\oksep}{\mathcal{O}_{K^{\mathrm{sep}}}}

\newcommand{\Tcrys}{T_{\mathrm{crys}}^{*}}

\newcommand{\TSG}{T^{*}_{\SG}}

\newcommand{\okp}{\mathcal{O}_{K_1}}

\newcommand{\tokbar}{\tilde{\mathcal{O}}_{\Kbar}}

\newcommand{\PD}{\mathrm{DP}}

\newcommand{\Gr}{\mathrm{Gr}}
\newcommand{\Mod}{\mathrm{Mod}}

\newcommand{\SG}{\mathfrak{S}}
\newcommand{\SGm}{\mathfrak{M}}
\newcommand{\SGn}{\mathfrak{N}}

\newcommand{\ModSGr}{\mathrm{Mod}_{/\mathfrak{S}_1}^{r,\phi}}
\newcommand{\ModSGrfr}{\mathrm{Mod}_{/\mathfrak{S}}^{r,\phi}}
\newcommand{\ModSGrinf}{\mathrm{Mod}_{/\mathfrak{S}_\infty}^{r,\phi}}

\newcommand{\ModSGf}{\mathrm{Mod}_{/\mathfrak{S}_1}^{1,\phi}}
\newcommand{\ModSGffr}{\mathrm{Mod}_{/\mathfrak{S}}^{1,\phi}}
\newcommand{\ModSGfinf}{\mathrm{Mod}_{/\mathfrak{S}_\infty}^{1,\phi}}

\newcommand{\ModSr}{\mathrm{Mod}_{/S_1}^{r,\phi}}
\newcommand{\ModSrfr}{\mathrm{Mod}_{/S}^{r,\phi}}
\newcommand{\ModSrinf}{\mathrm{Mod}_{/S_\infty}^{r,\phi}}

\newcommand{\ModSffr}{\mathrm{Mod}_{/S}^{1,\phi}}

\newcommand{\OEp}{\mathcal{O}_{\mathcal{E}}}
\newcommand{\Eur}{\mathcal{E}^{\mathrm{ur}}}

\newcommand{\hEur}{\widehat{\mathcal{E}^\mathrm{ur}}}
\newcommand{\hOEur}{\mathcal{O}_{\widehat{{\mathcal{E}^\mathrm{ur}}}}}
\newcommand{\SGur}{\mathfrak{S}^\mathrm{ur}}

\newcommand{\uep}{\underline{\varepsilon}}
\newcommand{\upi}{\underline{\pi}}

\newcommand{\OcX}{\mathcal{O}_{\mathcal{X}}}
\newcommand{\OcY}{\mathcal{O}_{\mathcal{Y}}}
\newcommand{\cXsep}{\mathcal{X}^{\mathrm{sep}}}
\newcommand{\OcXsep}{\mathcal{O}_{\mathcal{X}^\mathrm{sep}}}

\newcommand{\bC}{\mathbb{C}}
\newcommand{\bD}{\mathbb{D}}

\newcommand{\bZ}{\mathbb{Z}}
\newcommand{\bQ}{\mathbb{Q}}
\newcommand{\bF}{\mathbb{F}}

\newcommand{\sF}{\mathscr{F}}

\newcommand{\frt}{\mathfrak{t}}

\begin{document}

\title[Ramification correspondence of finite flat group schemes]{Ramification correspondence of finite flat group schemes over equal and mixed characteristic local fields}
\author{Shin Hattori}
\date{\today}
\email{shin-h@math.kyushu-u.ac.jp}
\address{Faculty of Mathematics, Kyushu University}
\thanks{Supported by Grant-in-Aid for Young Scientists B-21740023.}

\begin{abstract}
Let $p>2$ be a rational prime, $k$ be a perfect field of characteristic $p$ and $K$ be a finite totally ramified extension of the fraction field of the Witt ring of $k$. Let $\cG$ and $\cH$ be finite flat commutative group schemes killed by $p$ over $\okey$ and $k[[u]]$, respectively. In this paper, we show the ramification subgroups of $\cG$ and $\cH$ in the sense of Abbes-Saito are naturally isomorphic to each other when they are associated to the same Kisin module.
\end{abstract}

\maketitle

\section{Introduction}\label{intro}

Let $p$ be a rational prime, $k$ be a perfect field of characteristic $p$, $W=W(k)$ be the Witt ring of $k$ and $K$ be a finite totally ramified extension of $\Frac(W)$ of degree $e$. Let $\phi$ denote the Frobenius endomorphism of $W$. We fix once and for all an algebraic closure $\Kbar$ of $K$, a uniformizer $\pi$ of $K$ and a system of its $p$-power roots $\{\pi_n\}_{n\in\bZ_{\geq 0}}$ in $\Kbar$ with $\pi_0=\pi$ and $\pi_n=\pi_{n+1}^p$. Put $K_n=K(\pi_n)$, $K_\infty=\cup_n K_n$, $G_K=\Gal(\Kbar/K)$ and $G_{K_\infty}=\Gal(\Kbar/K_\infty)$. By the theory of norm fields (\cite{Wi}), there exist a complete discrete valuation field $\cX\simeq k((u))$ of characteristic $p$ with residue field $k$ and an isomorphism of groups
\[
G_{K_\infty}\simeq G_\cX=\Gal(\cXsep/\cX),
\]
where $\cXsep$ is a separable closure of $\cX$. A striking feature of this isomorphism is its compatibility with the upper ramification subgroups of both sides up to a shift by the Herbrand function of $K_\infty/K$ (\cite[Corollaire 3.3.6]{Wi}).

On the other hand, Breuil (\cite{Br}) introduced linear algebraic data over a ring $S$, which are now called as Breuil modules, and proved a classification of finite flat (commutative) group schemes over $\okey$ via these data for $p>2$. He (\cite{Br_nonpub}, \cite{Br_AZ}) also simplified this classification by replacing these data by $\phi$-modules over $W[[u]]$, which are referred as Kisin modules since the latter classification was reproved and investigated further by Kisin (\cite{Ki_Fcrys}, \cite{Ki_BT}). Let us consider the case where finite flat group schemes are killed by $p$. Put $\SG_1=k[[u]]$, which is isomorphic to the $k$-algebra $\OcX$. We let $\phi$ also denote the absolute Frobenius endomorphism of the ring $\SG_1$. For a non-negative integer $r$, let $\ModSGr$ be the category of free $\SG_1$-modules $\SGm$ of finite rank endowed with a $\phi$-semilinear map $\phi_\SGm:\SGm\to\SGm$ such that the cokernel of the map $1\otimes \phi_\SGm:\SG_1\otimes_{\phi,\SG_1}\SGm\to \SGm$ is killed by $u^{er}$. Then there exists an anti-equivalence of categories $\cG(-)$ from $\ModSGf$ to the category of finite flat group schemes over $\okey$ killed by $p$. It is well-known that, for any $r$, we also have an anti-equivalence $\cH(-)$ from $\ModSGr$ to a category of finite flat generically etale group schemes over $\OcX$ killed by their Verschiebung (\cite{SGA3-7A}). Hence a correspondence between finite flat group schemes over $\okey$ and $\OcX$ is obtained, and if finite flat group schemes $\cG$ over $\okey$ and $\cH$ over $\OcX$ are in correspondence with each other, then their generic fiber Galois modules $\cG(\okbar)$ and $\cH(\OcXsep)$ are also in correspondence via the theory of norm fields. Namely, for an object $\SGm$ of $\ModSGf$, we have an isomorphism of $G_{K_\infty}$-modules 
\[
\cG(\SGm)(\okbar)|_{G_{K_\infty}}\to \cH(\SGm)(\OcXsep)
\]
(\cite[Proposition 1.1.13]{Ki_BT}). From this, we can show that the Galois modules $\cG(\SGm)(\okbar)$ and $\cH(\SGm)(\OcXsep)$ have exactly the same greatest upper ramification jump in the classical sense (see also \cite{Ab_period}). 

Besides the classical ramification theory of their generic fibers, finite flat group schemes over a complete discrete valuation ring have their own ramification theory, which was discovered by Abbes-Saito (\cite{AS1}, \cite{AS2}) and Abbes-Mokrane (\cite{AM}). Such a finite flat group scheme $\cG$ has filtrations of upper ramification subgroups $\{\cG^j\}_{j\in\bQ_{>0}}$ (\cite{AM}) and lower ramification subgroups $\{\cG_i\}_{i\in\bQ_{\geq 0}}$ (\cite{Fa}, \cite{Ha1}) as in the classical ramification theory of local fields. For simplicity, let $K$ be a complete discrete valuation field as above and consider a finite flat group scheme $\cG$ over $\okey$. Then, using the upper ramification filtration of $\cG$, we can bound the classical greatest upper ramification jump of the generic fiber $G_K$-module $\cG(\okbar)$ (\cite{Ha1}) and also describe completely the semi-simplification of the restriction to the inertia subgroup of this $G_K$-module (\cite{Ha3}). Moreover, the canonical subgroup of a possibly higher dimensional abelian scheme $A$ over $\okey$ is found in the upper ramification filtration of $A[p^n]$ (\cite{AM}, \cite{Ti}, \cite{Ti_HN}), while the canonical subgroup is also found in the Harder-Narasimhan filtration of $A[p^n]$ defined by Fargues (\cite{Fa_HN}, \cite{Fa}).

In this paper, we establish the following correspondence of the ramification filtrations between finite flat group schemes over $\okey$ and $\OcX$ which is similar to that of the classical ramification jumps of their generic fiber Galois modules stated above.

\begin{thm}\label{main}
Let $p>2$ be a rational prime and $K$ and $\cX$ be as before. Let $\SGm$ be an object of the category $\ModSGf$. Then the natural isomorphism of $G_{K_\infty}$-modules $\cG(\SGm)(\okbar)|_{G_{K_\infty}}\to \cH(\SGm)(\OcXsep)$
induces isomorphisms of the upper and the lower ramification subgroups
\begin{align*}
\cG(\SGm)^j(\okbar)|_{G_{K_\infty}}&\to \cH(\SGm)^j(\OcXsep),\\
\cG(\SGm)_i(\okbar)|_{G_{K_\infty}}&\to \cH(\SGm)_i(\OcXsep)
\end{align*}
for any $j\in\bQ_{>0}$ and $i\in\bQ_{\geq 0}$.
\end{thm}

This theorem enables us to reduce the study of ramification of finite flat group schemes over $\okey$ killed by $p$ to the case where the base is a complete discrete valuation ring of equal characteristic. This makes calculations of ramification of finite flat group schemes over $\okey$, for example as in \cite[Section 5]{Ha3}, much easier. We remark that, for the Harder-Narasimhan filtration, such a correspondence of filtrations of $\cG(\SGm)$ and $\cH(\SGm)$ follows easily from the definition.

A key idea to prove the theorem is to switch from the upper ramification filtration to the lower ramification filtration via Cartier duality, which the author learned from works of Tian (\cite{Ti}) and Fargues (\cite{Fa}). Let $\cG$ be a finite flat group scheme over $\okey$ killed by $p$ and $\cG^\vee$ be its Cartier dual. Then they showed that the upper ramification subgroup $\cG^j(\okbar)$ is the orthogonal subgroup of the lower ramification subgroup $(\cG^\vee)_i(\okbar)$ for some $i$ via the Cartier pairing
\[
\cG(\okbar)\times \cG^\vee(\okbar)\to \bZ/p\bZ(1).
\]
We prove a version of this theorem for the group scheme $\cH(\SGm)$ over $\OcX$. Since we are in characteristic $p$, usual Cartier dual does not preserve the generic etaleness of finite flat group schemes. Instead, we use a duality theory of Liu (\cite{Li_FC}) for Kisin modules. This requires us to check compatibilities of these two duality theories, though it is straightforward to carry out. Thus we reduce ourselves to proving the correspondence of the lower ramification subgroups of $\cG(\SGm)$ and $\cH(\SGm)$. This is a consequence of the fact that, up to the base changes from $K$ and $\cX$ to the extensions generated by $p$-th roots of their uniformizers $\pi$ and $u$, the schemes $\cG(\SGm)$ modulo $\pi^e$ and $\cH(\SGm)$ modulo $u^e$ become isomorphic to each other, not as group schemes but as pointed schemes. We prove this fact by using Breuil's explicit computation of the affine algebra of a finite flat group scheme over $\okey$ killed by $p$ in terms of its corresponding Breuil module (\cite[Section 3]{Br}), after showing that his classification of finite flat group schemes is compatible with the base change from $K$ to $K_n$.

It should be mentioned that a classification of finite flat group schemes via Kisin modules is also proved for $p=2$ by Kisin (\cite{Ki_2}) for the case of unipotent finite flat group schemes and independently by Kim (\cite{Kim_2}), Lau (\cite{Lau_2}) and Liu (\cite{Li_2}) for the general case. However, the author does not know whether a similar correspondence of ramification between characteristic zero and two holds for finite flat group schemes with non-trivial multiplicative parts.



\section{Review of Cartier duality theory for Kisin modules}\label{Cart}

Let $p>2$ be a rational prime and $K$ be a complete discrete valuation field of mixed characteristic $(0,p)$ with perfect residue field $k$, as in Section \ref{intro}. It is well-known that finite flat group schemes over $\okey$ killed by some $p$-power are classified by linear algebraic data, Breuil modules (\cite{Br}) or Kisin modules (\cite{Br_nonpub}, \cite{Br_AZ}, \cite{Ki_Fcrys}, \cite{Ki_BT}). For these data, corresponding notions of duality to Cartier duality for finite flat group schemes are introduced by Caruso and Liu (\cite{Ca_th}, \cite{Ca_D}, \cite{CL}, \cite{Li_FC}), which play key roles in the integral $p$-adic Hodge theory. In this section, we recall the definitions of these data and the theory of Cartier duality for Kisin modules.

\subsection{Breuil and Kisin modules}

Let $E(u)\in W[u]$ be the Eisenstein polynomial of the uniformizer $\pi$ over $W$. Put $F(u)=p^{-1}(u^e-E(u))$. This defines units in the rings $\SG=W[[u]]$ and $\SG_1=k[[u]]$. The $\phi$-semilinear continuous ring endomorphisms of these rings defined by $u\mapsto u^p$ are also denoted by $\phi$. Let $r$ be a non-negative integer. Then a Kisin module over $\SG$ of $E$-height $\leq r$ is an $\SG$-module $\SGm$ endowed with a $\phi$-semilinear map $\phi_\SGm:\SGm \to \SGm$ such that the cokernel of the map $1\otimes \phi_\SGm: \SG\otimes_{\phi,\SG}\SGm\to \SGm$ is killed by $E(u)^r$. We write $\phi_\SGm$ also as $\phi$ if there is no risk of confusion. A morphism of Kisin modules over $\SG$ is an $\SG$-linear map which is compatible with $\phi$'s of the source and the target. The Kisin modules over $\SG$ of $E$-height $\leq r$ form a category with an obvious notion of exact sequences. We let $\ModSGr$ ({\it resp.} $\ModSGrfr$) denote its full subcategory consisting of $\SGm$ which is free of finite rank over $\SG_1$ ({\it resp.} $\SG$). We also let $\ModSGrinf$ denote its full subcategory consisting of $\SGm$ such that $\SGm$ is a finite $\SG$-module which is $p$-power torsion and $u$-torsion free.

The categories of Kisin modules have a natural duality theory (\cite[Subsection 2.4]{CL}). For an object $\SGm$ of the category $\ModSGrinf$, we let $\SGm^\vee$ denote its dual object as in {\it loc. cit.}. We give here an explicit description of the duality theory of the category $\ModSGr$ for the convenience of the reader. Let $\SGm$ be an object of this category. By definition, the $\SG_1$-module $\SGm^\vee$ is $\Hom_{\SG}(\SGm,\SG_1)$. Choose a basis $e_1,\ldots, e_d$ of the free $\SG_1$-module $\SGm$ and let $e^{\vee}_{1},\ldots,e^{\vee}_{d}$ denote its dual basis. Define a matrix $A\in M_d(\SG_1)$ by
\[
\phi_{\SGm}(e_1,\ldots,e_d)=(e_1,\ldots,e_d)A.
\]
Put $c_0=p^{-1}E(0)\in W^\times$. Then the $\phi$-semilinear map $\phi_{\SGm^\vee}:\SGm^\vee\to\SGm^\vee$ is given by
\[
\phi_{\SGm^\vee}(e^{\vee}_{1},\ldots,e^{\vee}_{d})=(e^{\vee}_{1},\ldots,e^{\vee}_{d})(E(u)/c_0)^r({}^\mathrm{t}\!A)^{-1}.
\]

For $r<p-1$, we also have categories $\ModSr$, $\ModSrinf$ and $\ModSrfr$ of Breuil modules defined as follows. Let $S$ be the $p$-adic completion of the divided power envelope $W[u]^\PD$ of $W[u]$ with respect to the ideal $(E(u))$ and the compatibility condition with the canonical divided power structure on $pW$. The ring $S$ has a natural filtration $\Fil^iS$ defined as the closure in $S$ of the ideal generated by $E(u)^j/j!$ for integers $j\geq i$. The $\phi$-semilinear continuous ring endomorphism of $S$ defined by $u\mapsto u^p$ is also denoted by $\phi$. For $0\leq i\leq p-1$, we have $\phi(\Fil^iS)\subseteq p^iS$ and put $\phi_i=p^{-i}\phi|_{\Fil^iS}$. These filtration and $\phi_i$'s induce a similar structure on the ring $S_n=S/p^nS$. Put $c=\phi_1(E(u))\in S^\times$. Then we let ${}'\Mod_{/S}^{r,\phi}$ denote the category of $S$-modules $\cM$ endowed with an $S$-submodule $\Fil^r\cM$ containing $(\Fil^rS)\cM$ and a $\phi$-semilinear map $\phi_{r,\cM}:\Fil^r\cM\to \cM$ satisfying $\phi_{r,\cM}(s_r m)=c^{-r}\phi_r(s_r)\phi_{r,\cM}(E(u)^rm)$ for any $s_r\in \Fil^rS$ and $m\in\cM$. A morphism of this category is defined to be a homomorphism of $S$-modules compatible with $\Fil^r$'s and $\phi_r$'s. We drop the subscript $\cM$ of $\phi_{r,\cM}$ if no confusion may occur. The category ${}'\Mod_{/S}^{r,\phi}$ has an obvious notion of exact sequences. Its full subcategory consisting of $\cM$ such that $\cM$ is a free $S_1$-module of finite rank and the image $\phi_{r,\cM}(\Fil^r\cM)$ generates the $S$-module $\cM$ is denoted by $\ModSr$. We let $\ModSrinf$ denote the smallest full subcategory of ${}'\Mod_{/S}^{r,\phi}$ containing $\ModSr$ and stable under extensions, and $\ModSrfr$ denote the full subcategory consisting of $\cM$ such that $\cM$ is a free $S$-module of finite rank, the $S$-module $\cM/\Fil^r\cM$ is $p$-torsion free and the image $\phi_{r,\cM}(\Fil^r\cM)$ generates the $S$-module $\cM$.

The categories $\ModSGrinf$ and $\ModSrinf$ for $r<p-1$ are in fact equivalent. We define an exact functor $\cM_\SG:\ModSGrinf\to \ModSrinf$ by putting $\cM_\SG(\SGm)=S\otimes_{\phi,\SG}\SGm$ with
\begin{align*}
&\Fil^r\cM_\SG(\SGm)=\Ker(S\otimes_{\phi,\SG}\SGm\overset{1\otimes\phi}{\to}(S/\Fil^rS)\otimes_{\SG}\SGm),\\
&\phi_r:\Fil^r\cM_\SG(\SGm)\overset{1\otimes \phi}{\to}\Fil^rS\otimes_{\SG}\SGm\overset{\phi_r\otimes 1}{\to} S\otimes_{\phi,\SG}\SGm=\cM_\SG(\SGm).
\end{align*}
Then the functor $\cM_\SG$ is an equivalence of categories (\cite[Theorem 2.3.1]{CL}). Similarly, we have an equivalence of categories $\ModSGrfr\to \ModSrfr$ (\cite[Theorem 2.2.1]{CL}), which is denoted also by $\cM_\SG$.


\subsection{The associated Galois representations and duality}\label{DuGal}

Next we recall constructions of the associated Galois representations to Breuil and Kisin modules and their duality theories. Let $v_K$ be the valuation on $K$ which is normalized as $v_K(\pi)=1$ and we extend it naturally to $\Kbar$. Set $\tokbar=\okbar/p\okbar$ and $\bC$ to be the completion of $\Kbar$. Consider the ring
\[
R=\varprojlim (\tokbar \leftarrow \tokbar \leftarrow \cdots),
\]
where the transition maps are defined by $x\mapsto x^p$. For an element $x=(x_0,x_1,\ldots)\in R$ with $x_i\in\tokbar$, we put $x^{(m)}=\lim_{n\to\infty}\hat{x}_{n+m}^{p^n}\in\cO_\bC$, where $\hat{x}_i$ is a lift of $x_i$ in $\okbar$. This is independent of the choice of lifts. Then the ring $R$ is a complete valuation ring of characteristic $p$ with valuation $v_R(x)=v_K(x^{(0)})$. We put $m_R^{\geq i}=\{x\in R\mid v_R(x)\geq i\}$ and similarly for $m_R^{>i}$. Define an element $\upi$ of $R$ by $\upi=(\pi,\pi_1,\pi_2,\ldots)$, where we abusively write $\pi_n$ also for its image in $\tokbar$. The $W$-algebras $R$ and $W(R)$ have natural $\SG$-algebra structures defined by the continuous map $\SG\to W(R)$ which sends $u$ to the Teichm\"uller lift $[\upi]$ of the element $\upi$. Note that the identification $G_{K_\infty}\simeq G_\cX$ stated in Section \ref{intro} is given by the action of $G_{K_\infty}$ on the ring $R$ and the inclusion $k[[u]]\to R$ defined by $u\mapsto \upi$ (\cite[Subsection 3.3]{Br_AZ}).

We set $\OEp$ to be the $p$-adic completion of $\SG[1/u]$ and put $\cE=\Frac(\OEp)$. We extend the inclusion $\SG\to W(R)$ to an inclusion $\OEp\to W(\Frac(R))$. The maximal unramified extension of $\cE$ in the field $W(\Frac(R))[1/p]$ is denoted by $\Eur$ and its closure in the same field by $\hEur$. We put $\SGur=\hOEur\cap W(R)$ inside the ring $W(\Frac(R))$. The Galois group $G_{K_\infty}$ acts naturally on the $\SG$-algebra $\SGur$. For an object $\SGm\in \ModSGrinf$, its associated $G_{K_\infty}$-module $\TSG(\SGm)$ is by definition
\[
\TSG(\SGm)=\Hom_{\SG,\phi}(\SGm, \bQ_p/\bZ_p\otimes \SGur).
\]
If the $\SG$-module $\SGm$ is killed by $p^n$, then we have a natural identification $\TSG(\SGm)\simeq \Hom_{\SG,\phi}(\SGm,W_n(R))$ (\cite[Proposition 1.8.3]{F_PhiGamma}). Similarly, for an object $\SGm$ of the category $\ModSGrfr$, its associated $G_{K_\infty}$-module is defined as 
\[
\TSG(\SGm)=\Hom_{\SG,\phi}(\SGm, \SGur).
\]

Consider the natural $W$-algebra surjection $\theta: W(R)\to \mathcal{O}_{\bC}$ defined by
\[
\theta((z_0,z_1,\ldots))=\sum_{i=0}^{\infty}p^i z_{i}^{(i)},
\]
where $z_i$ is an element of $R$. The $p$-adic completion of the divided power envelope of $W(R)$ with respect to the ideal $\Ker(\theta)$ is denoted by $\Acrys$. The ring $\Acrys$ has a Frobenius endomorphism $\phi$ and a $G_K$-action induced by those of $R$, and also a filtration induced by the divided power structure. The $W$-algebra homomorphism $W[u]\to W(R)$ defined by $u\mapsto [\upi]$ induces a map $S\to \Acrys$, by which we consider the ring $\Acrys$ as an $S$-algebra. For $0\leq r\leq p-2$, the ring $\Acrys$ has a natural structure as an object of ${}'\Mod_{/S}^{r,\phi}$ by putting $\phi_r=p^{-r}\phi|_{\Fil^r\Acrys}$. 

Let $\cM$ be an object of $\ModSrinf$. Then we also have the associated $G_{K_\infty}$-module
\[
\Tcrys(\cM)=\Hom_{S,\Fil^r,\phi_r}(\cM,\bQ_p/\bZ_p\otimes \Acrys).
\]
If $\cM$ is killed by $p$, then we have a natural identification
\[
\Tcrys(\cM)\simeq \Hom_{S,\Fil^r,\phi_r}(\cM,R^\PD),
\]
where $R^\PD$ is the divided power envelope of $R$ with respect to the ideal $m_R^{\geq e}$ and we identify this ring with $\Acrys/p\Acrys$. Similarly, for an object $\cM$ of the category $\ModSrfr$, we put
\[
\Tcrys(\cM)=\Hom_{S,\Fil^r,\phi_r}(\cM,\Acrys).
\]
The functors $\TSG$ and $\Tcrys$ from $\ModSGrinf$ and $\ModSrinf$ to the category of $G_{K_\infty}$-modules are exact. For an object $\SGm$ of $\ModSGr$, we have the equality $\dim_{\bF_p}(\TSG(\SGm))=\rank_{\SG_1}(\SGm)$ and a similar assertion also holds for $\Tcrys$. Then, for an object $\SGm$ of the category $\ModSGrinf$ or $\ModSGrfr$, we have an isomorphism of $G_{K_\infty}$-modules 
\[
\TSG(\SGm)\to\Tcrys(\cM_\SG(\SGm))
\]
defined by $f \mapsto (s\otimes m\mapsto s\phi(f(m)))$ (\cite[Lemma 3.3.4]{Li_BC}).

To describe the duality theory for $\TSG(\SGm)$, let us also fix a system $\{\zeta_{p^n}\}_{n\in\bZ_{\geq 0}}$ of $p$-power roots of unity in $\Kbar$ such that $\zeta_1=1$, $\zeta_p\neq 1$ and $\zeta_{p^n}=\zeta_{p^{n+1}}^p$, and set an element $\uep\in R$ to be $\uep=(1, \zeta_p, \zeta_{p^2},\ldots)$ with an abusive notation as before. Define an element $t$ of $\Acrys$ by 
\[
t=\log([\uep])=\sum_{i=1}^{\infty}(-1)^{i-1}\frac{([\uep]-1)^i}{i}.
\]
Put $c_0=p^{-1}E(0)\in W^\times$ and
\[
\lambda=\prod_{i=1}^{\infty}\phi^i(E(u)/E(0)) \in S^\times.
\]
Let $\SG(r)=\SG\mathbf{e}$ be the object of the category $\ModSGffr$ of rank one with a basis $\mathbf{e}$ satisfying $\phi(\mathbf{e})=(E(u)/c_0)^r\mathbf{e}$. Similarly, we define an object $S(r)$ of the category $\ModSffr$ of rank one by $S(r)=S\mathbf{e}=\Fil^r S(r)$ and $\phi_r(\mathbf{e})=\mathbf{e}$. Then the Breuil module $\cM_\SG(\SG(r))$ is isomorphic to $S(r)$ via the multiplication by $\lambda^{-r}$. Thus we have isomorphisms of $G_{K_\infty}$-modules
\[
\TSG(\SG(r))\to \Tcrys(\cM_\SG(\SG(r)))\to \Tcrys(S(r))=\bZ_p(\mathbf{e}\mapsto t^r).
\]
Their composite is given by $f\mapsto (\mathbf{e}\mapsto \lambda^r\phi(f(\mathbf{e})))$. Set an element $\frt\in\Acrys$ to be
\[
\frt=(\frac{E(u)}{c_0})^{-1}\lambda^{-1}t.
\]
Then the element $\frt$ is contained in the subring $\SGur$ (\cite[Subsection 3.2]{Li_FC}).

Let $\SGm$ be an object of the category $\ModSGrfr$ and $\SGm^\vee$ be its dual object. Then as in \cite[Subsection 2.4]{CL}, 
the evaluation map $\SGm\times \SGm^\vee \to \SG(r)$ induces a natural perfect pairing of $G_{K_\infty}$-modules
\[
\langle\ ,\ \rangle_\SGm:\TSG(\SGm)\times \TSG(\SGm^\vee) \to \TSG(\SG(r))\simeq \bZ_p\mathfrak{t}^r\subseteq \SGur,
\]
which gives an isomorphism of $G_{K_\infty}$-modules
\[
\TSG(\SGm^\vee) \to \Hom_{\bZ_p}(\TSG(\SGm),\bZ_p\mathfrak{t}^r).
\]
We also have a similar perfect pairing $\langle\ ,\ \rangle_\SGm$ for the category $\ModSGrinf$ (\cite{CL}).



\section{Cartier duality for upper and lower ramification subgroups}\label{ULdual}

Let $\cG$ be a finite flat generically etale (commutative) group scheme over the ring of integers of a complete discrete valuation field. Abbes-Mokrane (\cite{AM}) initiated a study of ramification of $\cG$ using a ramification theory of Abbes-Saito (\cite{AS1}, \cite{AS2}). As in the classical ramification theory of local fields, $\cG$ has upper and lower ramification subgroups (\cite{Fa}, \cite{Ha1}). When the base field is of mixed characteristic, Tian proved that the upper and the lower ramification subgroups correspond to each other via usual Cartier duality if $\cG$ is killed by $p$ (\cite{Ti}), and Fargues gave a much simpler proof of this theorem (\cite{Fa}). In this section, after briefly recalling the ramification theory of finite flat group schemes, we show a variant of Tian's theorem for a complete discrete valuation field of equal characteristic $p$ with perfect residue field, using the duality techniques presented in the previous section instead of Cartier duality of finite flat group schemes.

\subsection{Ramification theory of finite flat group schemes}

In this subsection, we let $K$ denote a complete discrete valuation field, $\pi$ a uniformizer of $K$, $\okey$ the ring of integers and $\Ksep$ a separable closure of $K$. Let $v_K$ be the valuation of $K$ normalized as $v_K(\pi)=1$ and extend it naturally to $\Ksep$. We put $m_{\Ksep}^{\geq i}=\{x\in\oksep\mid v_K(x)\geq i\}$
and similarly for $m_{\Ksep}^{>i}$. We also put $G_K=\Gal(\Ksep/K)$.

Let $\cG=\Spec(B)$ be a finite flat generically etale group scheme over $\okey$. Then $\cG$ is locally of complete intersection over $\okey$ (\cite[Proposition 2.2.2]{Br}) and we have a natural surjection of $G_K$-modules $\cG(\oksep)\to \sF^j(B)$ for $j\in \bQ_{>0}$, where $\sF^j(B)$ is the set of geometric connected components of the $j$-th tubular neighborhood of $B$ with a group structure induced by that of $\cG$ (\cite[Subsection 2.3]{AM}). We define the $j$-th upper ramification subgroup $\cG^j$ of $\cG$ for $j\in \bQ_{>0}$ to be the scheme-theoretic closure in $\cG$ of (the finite subgroup scheme of $\cG\times\Spec(K)$ associated to) the kernel of this surjection. On the other hand, for $i\in \bQ_{\geq 0}$, the scheme-theoretic closure in $\cG$ of the kernel of the natural homomorphism $\cG(\oksep)\to \cG(\oksep/m_{\Ksep}^{\geq i})$ is denoted by $\cG_i$ and called the $i$-th lower ramification subgroup of $\cG$. In particular, we have the equality 
\[
\cG_i(\oksep)=\Ker(\cG(\oksep)\to \cG(\oksep/m_{\Ksep}^{\geq i})).
\]
We also put
\[
\cG^{j+}(\oksep)=\bigcup_{j'>j}\cG^{j'}(\oksep),\quad \cG_{i+}(\oksep)=\bigcup_{i'>i}\cG_{i'}(\oksep)
\]
and set $\cG^{j+}$ and $\cG_{i+}$ to be their scheme-theoretic closures in $\cG$, respectively.

As in the classical case, the upper ({\it resp.} lower) ramification subgroups are compatible with quotients ({\it resp.} subgroups). Namely, for a faithfully flat homomorphism $\cG\to \cG''$ of finite flat group schemes over $\okey$, the image of $\cG^j(\oksep)$ in $\cG''(\oksep)$ coincides with $(\cG'')^j(\oksep)$ (\cite[Lemme 2.3.2]{AM}). From the definition, we also see that for a closed immersion $\cG'\to \cG$ of finite flat group schemes over $\okey$, the subgroup $\cG'(\oksep)\cap \cG_i(\oksep)$ coincides with $(\cG')_i(\oksep)$. In addition, for a finite extension $L/K$ of relative ramification index $e'$, we have natural isomorphisms
\[
(\oel\times_{\okey}\cG)^{je'}\to \oel\times_{\okey}\cG^j,\quad (\oel\times_{\okey}\cG)_{ie'}\to \oel\times_{\okey}\cG_i
\]
of finite flat group schemes over $\oel$.

Suppose that $K$ is of mixed characteristic $(0,p)$ and $\cG$ is killed by $p^n$. Then we have $\cG^j=0$ for $j>e(n+1/(p-1))$, where $e$ is the absolute ramification index of $K$ (\cite[Theorem 7]{Ha1}). Let $\cG^\vee$ be the Cartier dual of $\cG$ and consider the Cartier pairing
\[
\langle\ ,\ \rangle_\cG:\cG(\okbar)\times \cG^\vee(\okbar) \to \bZ/p^n\bZ(1).
\]
When $\cG$ is killed by $p$, we have the following duality theorem for the upper and the lower ramification subgroups of $\cG$ (\cite[Theorem 1.6]{Ti}, \cite[Proposition 6]{Fa}).
\begin{thm}\label{TF}
Let $K$ be a complete discrete valuation field of mixed characteristic $(0,p)$ with absolute ramification index $e$ and $\cG$ be a finite flat group scheme over $\okey$ killed by $p$. For $j\leq pe/(p-1)$, we have an equality
\[
\cG^j(\okbar)^{\bot}=(\cG^\vee)_{l(j)+}(\okbar)
\]
of subgroups of $\cG^\vee(\okbar)$, where $\bot$ means the orthogonal subgroup with respect to the pairing $\langle\ ,\ \rangle_\cG$ and $l(j)=e/(p-1)-j/p$.
\end{thm}


\subsection{Kisin modules and finite flat group schemes of equal characteristic}

Consider the complete discrete valuation field $\cX=k((u))$ with uniformizer $u$ and perfect residue field $k$. We embed $\OcX=\SG_1=k[[u]]$ into $R$ by $u\mapsto \upi$ as before. Then $R$ is the completion of the ring of integers of an algebraic closure of $\cX$. We let $\cXsep$ denote the separable closure of $\cX$ in $\Frac(R)$. 

Let $\cY$ be a finite extension of $\cX$ in $\Frac(R)$ of relative ramification index $e'$ and let $\phi$ also denote the absolute Frobenius endomorphism of $\cY$. We identify $\OcY$ with $l[[v]]$ for a finite extension $l$ of $k$. Define a category $\Mod^{r,\phi}_{/\OcY}$ to be the category of free $\OcY$-modules $\SGn$ of finite rank endowed with a $\phi$-semilinear map $\phi_\SGn:\SGn\to\SGn$ such that the cokernel of the map $1\otimes \phi_\SGn:\OcY\otimes_{\phi,\OcY}\SGn\to \SGn$ is killed by $u^{er}$.

For an object $\SGm$ of $\Mod^{r,\phi}_{/\OcX}=\ModSGr$, we consider the $\OcY$-module $\OcY\otimes_{\OcX}\SGm$ with the $\phi$-semilinear map $\phi\otimes \phi_{\SGm}$ as an object of $\Mod^{r,\phi}_{/\OcY}$. This defines a base change functor $\OcY\otimes_{\OcX}-:\ModSGr\to \Mod^{r,\phi}_{/\OcY}$. We also have a dual object $\SGn^\vee$ for an object $\SGn$ of the category $\Mod^{r,\phi}_{/\OcY}$, which is defined similarly to the duality theory of $\ModSGr$. Moreover, for an object $\SGm$ of $\ModSGr$, we have a natural isomorphism
\[
\OcY\otimes_{\OcX}\SGm^\vee \to (\OcY\otimes_{\OcX}\SGm)^\vee
\]
of $\Mod^{r,\phi}_{/\OcY}$, by which we identify both sides.

For a finite flat group scheme $\cH$ over a base scheme of characteristic $p$, we let $F_\cH$ and $V_\cH$ denote the Frobenius and the Verschiebung of $\cH$, respectively. We say that a finite flat group scheme $\cH$ over $\OcY$ is $v$-height $\leq s$ if its Verschiebung $V_\cH$ is zero and the cokernel of the natural map
\[
V_{\cH^\vee}:\OcY\otimes_{\phi,\OcY}\Lie(\cH^\vee)\to \Lie(\cH^\vee)
\]
is killed by $v^{s}$. The category of finite flat group schemes over $\OcY$ of $v$-height $\leq s$ is denoted by $\cC_{\OcY}^{\leq s}$. Then we have an anti-equivalence of categories
\[
\cH_\cY(-): \Mod^{r,\phi}_{/\OcY} \to \cC_{\OcY}^{\leq ee'r}
\]
(\cite[Th\'eor\`eme 7.4]{SGA3-7A}). The group scheme $\cH_\cY(\SGn)$ is defined as a functor over $\OcY$ by 
\[
\mathfrak{A} \mapsto \Hom_{\OcY,\phi}(\SGn,\mathfrak{A}),
\]
where we consider an $\OcY$-algebra $\mathfrak{A}$ as a $\phi$-module over $\OcY$ with the absolute Frobenius endomorphism of $\mathfrak{A}$. If we choose a basis $e_1,\ldots,e_d$ of $\SGn$ and take a matrix $A=(a_{i,j})\in M_d(\OcY)$ satisfying
\[
\phi(e_1,\ldots,e_d)=(e_1,\ldots,e_d)A,
\]
then $\cH_\cY(\SGn)$ is isomorphic to the additive group scheme over $\OcY$ defined by the system of equations 
\[
X_i^p-\sum_{j=1}^{d}a_{j,i}X_j=0\ \ (i=1,\ldots,d).
\]
For an object $\SGm$ of $\ModSGr$, we also have a natural isomorphism
\[
\OcY\times_{\OcX} \cH_\cX(\SGm) \to \cH_\cY(\OcY\otimes_{\OcX}\SGm)
\]
of finite flat group schemes over $\OcY$. We drop the subscript $\cY$ of $\cH_\cY$ if there is no risk of confusion.

The following lemma is a variant of the scheme-theoretic closure for finite flat group schemes. 
\begin{lem}\label{closure}
Let $\SGm$ be an object of $\ModSGr$ and $L$ be a $G_\cX$-stable subgroup of $\TSG(\SGm)$. Then there exists a surjection $\SGm \to \SGm''$ of $\ModSGr$ such that the image of the corresponding injection $\TSG(\SGm'')\to\TSG(\SGm)$ coincides with $L$. A surjection $\SGm\to\SGm''$ satisfying this property is unique up to a unique isomorphism.
\end{lem}
\begin{proof}
This follows from \cite[Lemma 2.3.6]{Li_FC}. Indeed, let $\Mod^{\phi}_{/\cX}$ denote the category of etale $\phi$-modules over $\cX$ (\cite[Section A1]{F_PhiGamma}). We have an equivalence of categories $T_{*}$ from $\Mod^{\phi}_{/\cX}$ to the category of finite $G_\cX$-modules over $\bF_p$ defined by $T_{*}(M)=(\cXsep\otimes_{\cX}M)^{\phi=1}$. For an object $M$ of $\Mod^{\phi}_{/\cX}$, we also put $T^{*}(M)=\Hom_{\cX,\phi}(M,\cXsep)$. Then the natural map
\[
T_{*}(M)\to \Hom_{\bF_p}(T^{*}(M),\bF_p)
\]
is an isomorphism of $G_\cX$-modules. Set $M=\cX\otimes_{\SG_1}\SGm$. We have a natural isomorphism of $G_\cX$-modules $\TSG(\SGm)\to T^{*}(M)$ and let $M''$ be the quotient of $M$ corresponding to the surjection 
\[
T_{*}(M)\to \Hom_{\bF_p}(T^{*}(M),\bF_p)\to \Hom_{\bF_p}(L,\bF_p).
\]
Then the Kisin module $\SGm''=\Img(\SGm\to M\to M'')$ satisfies the desired property.
\end{proof}

Since the finite flat group scheme $\cH(\SGm)$ is generically etale, the group $\cH(\SGm)(\OcXsep)$ can be identified with the group $\cH(\SGm)(R)$ and we have the $j$-th upper ramification subgroup $\cH(\SGm)^j(\OcXsep)=\cH(\SGm)^j(R)$ of $\cH(\SGm)$. We also have the $i$-th lower ramification subgroup
\[
\cH(\SGm)_i(\OcXsep)=\Ker(\cH(\SGm)(\OcXsep)\to \cH(\SGm)(\OcXsep/m_{\cXsep}^{\geq i})),
\]
which we identify with
\[
\cH(\SGm)_i(R)=\Ker(\cH(\SGm)(R)\to \cH(\SGm)(R/m_{R}^{\geq i}))
\]
by using the injection $\OcXsep/m_{\cXsep}^{\geq i}\to R/m_R^{\geq i}$.
Since $\cH(\SGm)(R)=\TSG(\SGm)$, the pairing $\langle\ ,\ \rangle_\SGm$ of Subsection \ref{DuGal} induces a perfect pairing
\[
\langle\ ,\ \rangle_\SGm: \cH(\SGm)(R)\times \cH(\SGm^\vee)(R)\to R.
\]
Then the main theorem of this section is the following.

\begin{thm}\label{ramdualp}
Let $\SGm$ be an object of $\ModSGr$. Then we have $\cH(\SGm)^j=0$ for $j> per/(p-1)$. Moreover, for $j\leq per/(p-1)$,
we have the equality
\[
\cH(\SGm)^j(R)^{\bot}=\cH(\SGm^\vee)_{l_r(j)+}(R)
\]
of subgroups of $\cH(\SGm^\vee)(R)$, where $\bot$ means the orthogonal subgroup with respect to the pairing $\langle\ ,\ \rangle_\SGm$ and $l_r(j)=er/(p-1)-j/p$.
\end{thm}
\begin{proof}
We proceed as in the proof of \cite[Proposition 6]{Fa}. Let $\cY$ be a finite separable extension of $\cX$ in $\cXsep$ of relative ramification index $e'$ and put $\SGm_{\cY}=\OcY\otimes_{\SG_1}\SGm$. Then we have a commutative diagram 
\[
\xymatrix{
\cH(\SGm)(R)\times\cH(\SGm^\vee)(R)\ar@<-6ex>[d]_{\wr}\ar@<+6ex>[d]_{\wr} \ar[r]^-{\langle\ ,\ \rangle_{\SGm}} & R \ar@{=}[d]\\
\cH_{\cY}(\SGm_{\cY})(R)\times \cH_{\cY}(\SGm_{\cY}^\vee)(R)\ar[r]_-{\langle\ ,\ \rangle_{\SGm_{\cY}}} & R,
}
\]
where $\langle\ ,\ \rangle_{\SGm_{\cY}}$ is a perfect pairing defined similarly to the pairing $\langle\ ,\ \rangle_{\SGm}$ and the vertical arrows are isomorphisms. Since $\cH(\SGm^\vee)$ is generically etale, after making a finite separable base change and replacing $e$ by $ee'$, we may assume that the $G_\cX$-action on $\cH(\SGm^\vee)(R)$ is trivial.

Let $x^\vee$ be an element of $\cH(\SGm^\vee)(R)$ and consider the surjection $\SGm^\vee\to \SGn$ of the category $\ModSGr$ corresponding to the subspace $\bF_p x^\vee\subseteq \cH(\SGm^\vee)(R)$ by Lemma \ref{closure}. Then we have the commutative diagram
\[
\xymatrix{
\cH(\SGm)(R)\times\cH(\SGm^\vee)(R)\ar@<-6ex>[d]\ar[r]^-{\langle\ ,\ \rangle_{\SGm}} & R \ar@{=}[d]\\
\cH(\SGn^\vee)(R)\times \cH(\SGn)(R)\ar@<-6ex>[u]\ar[r]_-{\langle\ ,\ \rangle_{\SGn}} & R.
}
\]
Thus, by the compatibility with quotients ({\it resp.} subgroups) of the upper ({\it resp.} lower) ramification subgroups, the theorem follows from Lemma \ref{onedim} below.
\end{proof}

\begin{lem}\label{onedim}
Let $\SGn$ be an object of $\ModSGr$ which is free of rank one over $\SG_1$. Then we have $\cH(\SGn^\vee)^j(R)=0$ if $j>per/(p-1)$ and $\cH(\SGn)_i(R)=0$ if $i>er/(p-1)$. For $j\leq per/(p-1)$, the subgroup $\cH(\SGn^\vee)^j(R)$ is zero if and only if $\cH(\SGn)_{l_r(j)+}(R)=\cH(\SGn)(R)$.
\end{lem}
\begin{proof}
Let $\mathbf{n}$ be a basis of $\SGn$ and $\mathbf{n}^\vee$ be its dual basis of $\SGn^\vee$. Put $\phi_{\SGn}(\mathbf{n})=u^{s}a\mathbf{n}$ with $0\leq s\leq er$ and $a\in\SG_1^\times$. Then we have $\phi_{\SGn^\vee}(\mathbf{n}^\vee)=u^{er-s}a'\mathbf{n}^\vee$ with some $a'\in\SG_1^\times$. Hence the defining equations of $\cH(\SGn)$ and $\cH(\SGn^\vee)$ are $X^p-u^saX=0$ and $X^p-u^{er-s}a'X=0$, respectively. By a calculation as in \cite[Section 3]{Ha1}, we see that
\begin{align*}
\cH(\SGn^\vee)^j(R)&=\left\{
\begin{array}{ll}
\cH(\SGn^\vee)(R) & (j\leq p(er-s)/(p-1))\\
0 & (j>p(er-s)/(p-1)),
\end{array}\right.\\
\cH(\SGn)_i(R)&=\left\{
\begin{array}{ll}
\cH(\SGn)(R) & (i\leq s/(p-1))\\
0 & (i>s/(p-1))
\end{array}\right.
\end{align*}
and the first assertion follows. Moreover, for $j\leq per/(p-1)$, we have $l_r(j)\geq 0$ and
\begin{align*}
\cH(\SGn^\vee)^j(R)=0 &\Leftrightarrow j>p(er-s)/(p-1)\\
&\Leftrightarrow l_r(j)<s/(p-1) \Leftrightarrow \cH(\SGn)_{l_r(j)+}(R)=\cH(\SGn)(R).
\end{align*}
\end{proof}

By the previous lemma, we also have the following corollary.

\begin{cor}\label{lowramboundp}
Let $\SGm$ be an object of $\ModSGr$. Then the $i$-th lower ramification subgroup $\cH(\SGm)_i$ vanishes for $i>er/(p-1)$.
\end{cor}
\begin{proof}
As in the proof of Theorem \ref{ramdualp}, we may assume that the $G_\cX$-module $\cH(\SGm)(R)$ is trivial. For $i>er/(p-1)$, take an element $x\in \cH(\SGm)_i(R)$ and consider the quotient $\SGm\to \SGn$ corresponding to the subspace $\bF_px\subseteq \cH(\SGm)(R)$. By Lemma \ref{onedim}, we have $\cH(\SGn)_i(R)=0$ and thus $x=0$.
\end{proof}

\begin{rmk}\label{lowrambound}
Let $K$ be a complete discrete valuation field of mixed characteristic $(0,p)$ and $\cG$ be a finite flat group scheme over $\okey$ killed by $p$. Then, by using the usual scheme-theoretic closure of finite group schemes, we can easily see that the $i$-th lower ramification subgroup $\cG_i$ vanishes for $i>e/(p-1)$, just as in the proof of Corollary \ref{lowramboundp}.
\end{rmk}




\section{Comparison of ramification}\label{Compar}

Let $p>2$ be a rational prime and $K$ be a complete discrete valuation field of mixed characteristic $(0,p)$ with perfect residue field $k$, as in Section \ref{Cart}. Then it is known that there exists an anti-equivalence $\cG(-)$ from the category $\ModSGf$ to the category of finite flat group schemes over $\okey$ killed by $p$. On the other hand, we also have the anti-equivalence $\cH(-):\ModSGf\to \cC_{\OcX}^{\leq e}$ defined in Section \ref{ULdual} and an isomorphism of $G_{K_\infty}$-modules
\[
\varepsilon_\SGm: \cG(\SGm)(\okbar)|_{G_{K_\infty}}\to \cH(\SGm)(R).
\]
In this section, we prove that this isomorphism is compatible with the upper and the lower ramification subgroups of both sides. For the proof, after recalling the definitions of the anti-equivalence $\cG(-)$ and the isomorphism $\varepsilon_\SGm$ (\cite{Br}, \cite{Ki_Fcrys}, \cite{Ki_BT}), we show that these are compatible with the base changes inside $K_\infty/K$ and dualities on both sides. Then, by the duality theorems presented in Section \ref{ULdual}, we reduce ourselves to comparing the lower ramification subgroups of both sides, which is achieved by constructing an isomorphism as pointed schemes between reductions of $\cG(\SGm)$ and $\cH(\SGm)$.


\subsection{Breuil-Kisin classification}\label{RecallModGr}

In this subsection, we briefly recall the classification of finite flat group schemes and $p$-divisible groups over $\okey$ due to Breuil and Kisin (\cite{Br_nonpub}, \cite{Br}, \cite{Br_AZ}, \cite{Ki_Fcrys}, \cite{Ki_BT}) and its properties. Their classification theorem is as follows.

\begin{thm}\label{BKclass}
\begin{enumerate}
\item There exists an anti-equivalence of categories $\cG(-)$ from $\ModSGffr$ to the category of $p$-divisible groups over $\okey$ with a natural isomorphism of $G_{K_\infty}$-modules
\[
\varepsilon_\SGm: T_p\cG(\SGm)|_{G_{K_\infty}}\to \TSG(\SGm),
\]
where $T_p\cG(\SGm)=\varprojlim_n \cG(\SGm)[p^n](\okbar)$ is the $p$-adic Tate module of the $p$-divisible group $\cG(\SGm)$.
\item There exists an anti-equivalence of categories $\cG(-)$ from $\ModSGfinf$ to the category of finite flat group schemes over $\okey$ killed by some $p$-power with a natural isomorphism of $G_{K_\infty}$-modules
\[
\varepsilon_\SGm: \cG(\SGm)(\okbar)|_{G_{K_\infty}} \to \TSG(\SGm).
\]
\item Let $\SGm$ be an object of the category $\ModSGfinf$ and take a resolution of Kisin modules
\[
0\to \SGm_1 \overset{f}{\to} \SGm_2\to \SGm \to 0,
\]
where $\SGm_i$ is an object of the category $\ModSGffr$. Put $\cG=\cG(\SGm)$ and $\Gamma_i=\cG(\SGm_i)$. Then we have an exact sequence of fppf sheaves
\[
0\to \cG \to \Gamma_2\overset{\cG(f)}{\to} \Gamma_1 \to 0
\]
which induces the commutative diagram with exact rows
\[
\xymatrix{
0 \ar[r] & T_p\Gamma_2 \ar[r]\ar[d]_{\varepsilon_{\SGm_2}} & T_p\Gamma_1 \ar[r]\ar[d]_{\varepsilon_{\SGm_1}} & \cG(\okbar) \ar[r]\ar[d]_{\varepsilon_{\SGm}} & 0\\
0 \ar[r] &  \TSG(\SGm_2) \ar[r] &  \TSG(\SGm_1) \ar[r] &  \TSG(\SGm) \ar[r]&  0.
}
\]
\end{enumerate}
\end{thm}

We have two definitions of the functor $\cG(-)$ and the isomorphism $\varepsilon_\SGm$. One is putting $\cG(-)=\Gr^\mathrm{B}(\cM_\SG(-))$ with the anti-equivalence $\Gr^\mathrm{B}(-)$ of \cite{Br}, as in \cite{Ki_BT}. The isomorphism $\varepsilon_\SGm$ is constructed in the proof of \cite[Proposition 1.1.13]{Ki_BT}. Then we can prove Theorem \ref{BKclass} (3) by using \cite[Proposition 2.1.3 and Theorem 2.3.4]{CL}. The other is given in \cite{Ki_Fcrys}, as follows. For a $p$-divisible group $\Gamma$ over $\okey$, set $\Mod^\mathrm{K}(\Gamma)$ to be the section $\bD^{*}(\Gamma)_{(S\to \okey)}$ of the contravariant Dieudonn\'{e} crystal $\bD^{*}(\Gamma)$ (\cite{BBM}) on the divided power thickening $S\to \okey$ defined by $u\mapsto \pi$. The $S$-module $\Mod^\mathrm{K}(\Gamma)$ is endowed with the natural Frobenius and the Hodge filtration. Then it is shown that this defines an anti-equivalence from the category of $p$-divisible groups over $\okey$ to the category $\ModSffr$ with a quasi-inverse $\Gr^\mathrm{K}$ (\cite[Proposition A.6]{Ki_Fcrys}). Put $\cG(-)=\Gr^\mathrm{K}(\cM_\SG(-))$ and define $\cG(-)$ for the category $\ModSGfinf$ by taking a resolution as in Theorem \ref{BKclass} (3). For an object $\SGm$ of the category $\ModSGffr$, put $\cM=\cM_\SG(\SGm)$ and $\Gamma=\cG(\SGm)$. Consider the divided power thickening $\Acrys\to \oc$ induced by the map $\theta:W(R)\to \oc$ and we let $\cF_{\Acrys}$ denote the section of a crystalline sheaf $\cF$ on this thickening. Then we define the isomorphism $\varepsilon_\SGm$ in this case to be the composite of the isomorphism $\TSG(\SGm)\to \Tcrys(\cM)$ and the natural isomorphism $\varepsilon_{\cM}: T_p\Gamma\to \Tcrys(\cM)$ induced by the isomorphism
\[
T_p\Gamma \to \Tcrys(\bD^{*}(\Gamma)_{(S\to \okey)})=\Hom_{S,\Fil^1,\phi_1}(\bD^{*}(\Gamma)_{\Acrys}, \bD^{*}(\bQ_p/\bZ_p)_{\Acrys})
\]
sending a homomorphism $g:\bQ_p/\bZ_p\to \Gamma$ of $p$-divisible groups over $\oc$ to $\bD^{*}(g)$ (\cite[Theorem 7]{Fal}). We can see that these two definitions are naturally isomorphic to each other.

We need an explicit construction of the isomorphism $\varepsilon_\SGm$ for an object $\SGm$ of $\ModSGf$ (\cite[Proposition 1.1.13]{Ki_BT}). Let $H_1$ be the divided power polynomial ring $R^\PD\langle u-\upi\rangle$ over $R^\PD$. Consider the $n$-th projection $\prjt_n: R\to \tokbar$. The map $\prjt_0$ induces a surjection $H_1=R^\PD\langle u-\upi\rangle \to \tokbar$ defined by $u\mapsto \pi$, which is a divided power thickening of $\tokbar$. The map $H_1\to R^\PD$ defined by $u\mapsto \upi$ induces an isomorphism of $G_{K_\infty}$-modules
\[
\cG(\SGm)(\okbar)=\Hom_{S,\Fil^1,\phi_1}(\cM,H_1)\to \Hom_{S,\Fil^1,\phi_1}(\cM,R^\PD)
\]
(see the proof of \cite[Lemme 5.3.1]{Br}) and the isomorphism $\varepsilon_\SGm$ is the composite of this map and the natural isomorphism $\TSG(\SGm)\to\Tcrys(\cM)$.

Let $\SGm$ be an object of $\ModSGf$ and put $\cM=\cM_\SG(\SGm)$. In \cite[Section 3]{Br}, Breuil gave an explicit description of the affine algebra $R_\cM$ of the finite flat group scheme $\cG(\SGm)=\Gr^\mathrm{B}(\cM)$ in terms of $\cM$. Let $e_1,\ldots,e_d$ be an adapted basis of $\cM$ (\cite[D\'{e}finition 2.1.2.6]{Br}) such that 
\begin{align*}
&\Fil^1\cM=u^{r_1}S_1e_1\oplus\cdots\oplus u^{r_d}S_1e_d+(\Fil^pS)\cM,\\
&\phi_1(u^{r_1}e_1,\ldots, u^{r_d}e_d)=(e_1,\ldots, e_d)G
\end{align*}
with a matrix $G\in GL_d(S_1)$ (note that we adopt the transpose of the notation in \cite{Br}). Put $\tilde{S}_1=k[u]/(u^{ep})$ and identify this $k$-algebra with $S_1/\Fil^pS_1$. Assume that the image $\tilde{G}$ of $G$ in $GL_d(\tilde{S}_1)$ is contained in the subgroup $GL_d(k[u^p]/(u^{ep}))$. Consider the ring homomorphisms
\[
k[u]/(u^{ep})\to \cO_{K_1}/p\cO_{K_1} \gets \cO_{K_1},
\]
where the first map is the $\phi^{-1}$-semilinear isomorphism defined by $u \mapsto \pi_1$. We choose a lift $G_\pi=(a_{i,j})$ of $\tilde{G}$ by this map to $GL_d(\okey)$ using the assumption on $G$ and put
\[
R_\cM=\frac{\okey[X_1,\ldots,X_d]}{\left(X_1^p+\frac{\pi^{e-r_1}}{F(\pi)}(\sum_{j=1}^{d}a_{j,1}X_j),\ldots,X_d^p+\frac{\pi^{e-r_d}}{F(\pi)}(\sum_{j=1}^{d}a_{j,d}X_j)\right)}.
\]
Then we have an isomorphism of $p$-adic formal schemes $\Spf(R_\cM)\simeq \cG(\SGm)$. The induced map
\[
\Psi: \Hom_{\okey\text{-alg.}}(R_\cM,\okbar)\to \cG(\SGm)(\okbar)=\Hom_{S,\Fil^1,\phi_1}(\cM,H_1)
\]
is defined as follows. Let $f:R_\cM\to \okbar$ be an element of the left-hand side. Set $x_i=f(X_i)$ and $\bar{x}_i$ to be its image in $\tokbar$. Let $(\tokbar)^\PD$ be the divided power envelope of the ring $\tokbar$ with respect to the ideal $m_{\Kbar}^{\geq e/p}$. Then the map $\prjt_1$ induces an isomorphism $\prjt_1:R^{\PD} \to (\tokbar)^\PD$. We let $y_i$ denote the inverse image of $\bar{x}_i$ by this isomorphism. Then $\Psi(f)$ is the unique element of the right-hand side which is congruent to the $S$-linear map $(e_i\mapsto y_i)$ modulo $\Fil^pH_1$ (see the proof of \cite[Proposition 3.1.5]{Br}). From this description, we see that the zero section of the group scheme $\Spec(R_\cM)$ is defined by $X_1=\cdots=X_d=0$.


\subsection{Compatibility with a base change and Cartier duality}\label{subsecBC}

The functor $\cG(-)$ is compatible with the base changes inside $K_\infty/K$. Moreover, Caruso proved its compatibility with Cartier duality (\cite{Ca_D}). In this subsection, we briefly present a proof of these facts, along with similar compatibilities of the isomorphism $\varepsilon_\SGm$. First we recall the following lemma.

\begin{lem}\label{ffthm}
The functor $\Gamma\mapsto T_p\Gamma|_{G_{K_\infty}}$ from the category of $p$-divisible groups over $\okey$ to the category of $p$-adically continuous $G_{K_\infty}$-representations is fully faithful.
\end{lem}
\begin{proof}
This follows from Tate's theorem (\cite[Subsection 4.2, Corollary 1]{Ta}) and \cite[Theorem 3.4.3]{Br_AZ}.
\end{proof}

Let us show the compatibility with the base change. Put $\SG'=W[[v]]$ and let $\phi:\SG'\to \SG'$ denote the natural $\phi$-semilinear map defined by $v\mapsto v^p$. Note that the polynomial $E'(v)=E(v^{p^n})$ is the Eisenstein polynomial of the uniformizer $\pi_n\in K_n$. Consider the categories $\mathrm{Mod}^{1,\phi}_{/\SG'_1}$, $\mathrm{Mod}^{1,\phi}_{/\SG'_\infty}$ and $\mathrm{Mod}^{1,\phi}_{/\SG'}$ of Kisin modules over $\SG'$ of $E'$-height $\leq 1$. Using the $W$-algebra homomorphism $\SG'\to W(R)$ defined by $v\mapsto [\upi^{1/p^n}]$, we define a similar functor $T^{*}_{\SG'}$ to $\TSG$. The homomorphism of $W$-algebras $\SG\to \SG'$ defined by $u\mapsto v^{p^n}$ induces natural functors $(-)':\ModSGfinf\to \mathrm{Mod}^{1,\phi}_{/\SG'_\infty}$ and $(-)':\ModSGffr\to \mathrm{Mod}^{1,\phi}_{/\SG'}$ by
\[
\SGm\mapsto \SGm'=\SG'\otimes_\SG \SGm,\quad \phi_{\SGm'}=\phi\otimes \phi_{\SGm}.
\]
Then we have a natural isomorphism of $G_{K_\infty}$-modules $\TSG(\SGm)\to T^{*}_{\SG'}(\SGm')$. On the other hand, these categories classify finite flat group schemes killed by some $p$-power and $p$-divisible groups over $\okeyn$ by Theorem \ref{BKclass}, and we let $\cG'(-)$ denote the anti-equivalences of the theorem over $\okeyn$. 

\begin{prop}\label{BCK}
Let $\SGm$ be an object of the category $\ModSGffr$ ({\it resp.} $\ModSGfinf$) and $\SGm'$ be the associated object of the category $\mathrm{Mod}^{1,\phi}_{/\SG'}$ ({\it resp.} $\mathrm{Mod}^{1,\phi}_{/\SG'_\infty}$). Then there exists a natural isomorphism
\[
\cG'(\SGm') \to \okeyn\times_{\okey} \cG(\SGm)
\]
of $p$-divisible groups ({\it resp.} finite flat group schemes) over $\okeyn$ which respectively makes the following diagrams commutative:
\[
\xymatrix{
T_p\cG(\SGm)|_{G_{K_\infty}} \ar[r]_{\sim}\ar[d]_{\wr} & T_p\cG'(\SGm')|_{G_{K_\infty}}\ar[d]_{\wr} \\
\TSG(\SGm) \ar[r]_{\sim} & T^{*}_{\SG'}(\SGm'),
}
\]
\[
\xymatrix{
\cG(\SGm)(\okbar)|_{G_{K_\infty}} \ar[r]_{\sim}\ar[d]_{\wr} & \cG'(\SGm')(\okbar)|_{G_{K_\infty}}\ar[d]_{\wr} \\
\TSG(\SGm) \ar[r]_{\sim} & T^{*}_{\SG'}(\SGm').
}
\]
\end{prop}
\begin{proof}
The assertion for the category $\ModSGffr$ follows from Lemma \ref{ffthm}, and this implies the assertion for $\ModSGfinf$ by taking a resolution of $\SGm$ as in Theorem \ref{BKclass} (3). Note that the isomorphism $\cG'(\SGm') \to \okeyn\times_{\okey} \cG(\SGm)$ is independent of the choice of a resolution (\cite[Lemma 2.3.4]{Ki_Fcrys}).
\end{proof}

Next we show the compatibility with Cartier duality. For the element $\mathfrak{t}\in \SGur$, we let $\mathfrak{t}_n$ denote its image in the ring $\SGur_n=\SGur/p^n\SGur$.

\begin{prop}\label{CarGM}
Let $\SGm$ be an object of the category $\ModSGffr$ ({\it resp.} $\ModSGfinf$) and $\SGm^\vee$ be its dual object. Then there exists a natural isomorphism $\cG(\SGm)^\vee\to \cG(\SGm^\vee)$ of $p$-divisible groups ({\it resp.} finite flat group schemes) over $\okey$ such that the induced map 
\begin{align*}
&\delta_\SGm: T_p(\cG(\SGm)^\vee)\to T_p\cG(\SGm^\vee)\overset{\varepsilon_{\SGm^\vee}}{\to} \TSG(\SGm^\vee)\\
(\text{{\it resp. }}&\delta_\SGm: \cG(\SGm)^\vee(\okbar)\to \cG(\SGm^\vee)(\okbar) \overset{\varepsilon_{\SGm^\vee}}{\to} \TSG(\SGm^\vee))
\end{align*}
respectively makes the following diagrams of $G_{K_\infty}$-modules commutative:
\[
\xymatrix{
T_p\cG(\SGm)\times T_p(\cG(\SGm)^\vee)\ar@<-6ex>[d]^{\wr}_{\varepsilon_\SGm}\ar@<+6ex>[d]^{\wr}_{\delta_\SGm} \ar[r]& \bZ_p(1) \ar[d]^{\wr}\\
\TSG(\SGm)\times\TSG(\SGm^\vee)\ar[r]_-{\langle\ ,\ \rangle_\SGm} & \bZ_p\mathfrak{t},
}
\]
\[
\xymatrix{
\cG(\SGm)(\okbar)\times \cG(\SGm)^\vee(\okbar)\ar[r]\ar@<-6ex>[d]^{\wr}_{\varepsilon_\SGm}\ar@<+6ex>[d]^{\wr}_{\delta_\SGm} & \bZ/p^n\bZ(1) \ar[d]^{\wr}\\
\TSG(\SGm)\times\TSG(\SGm^\vee)\ar[r]_-{\langle\ ,\ \rangle_\SGm} & (\bZ/p^n\bZ)\mathfrak{t}_n.
}
\]
Here the top arrows are the Cartier pairings and the right vertical arrows are the isomorphisms induced by $(\zeta_{p^n})_{n\in\bZ_{\geq 0}}\mapsto \mathfrak{t}$.
\end{prop}
\begin{proof}
For an object $\SGm$ of the category $\ModSGffr$, define $\delta_\SGm: T_p(\cG(\SGm)^\vee)\to \TSG(\SGm^\vee)$ to be the unique isomorphism which fits into the commutative diagram of $G_{K_\infty}$-modules
\[
\xymatrix{
T_p(\cG(\SGm)^\vee) \ar[r] \ar@{.>}[d] & \Hom_{\bZ_p}(T_p(\cG(\SGm)),\bZ_p(1)) \ar[d] \\
\TSG(\SGm^\vee) \ar[r] & \Hom_{\bZ_p}(\TSG(\SGm), \bZ_p\mathfrak{t}),
}
\]
where the top ({\it resp.} the bottom) arrow is induced by the Cartier pairing ({\it resp.} the pairing $\langle\ ,\ \rangle_\SGm$) and the right vertical arrow is induced by $\varepsilon_\SGm$ and the natural isomorphism $\bZ_p(1)\to \bZ_p\mathfrak{t}$ as in the proposition. Then the isomorphism $\varepsilon_{\SGm^\vee}^{-1}\circ \delta_\SGm$ defines the desired isomorphism $\cG(\SGm)^\vee \to \cG(\SGm^\vee)$ uniquely by Lemma \ref{ffthm}.  

Next let $\SGm$ be an object of the category $\ModSGfinf$. Take a resolution $0\to \SGm_1\to \SGm_2\to \SGm\to 0$ of $\SGm$ as in Theorem \ref{BKclass} (3). Then this induces a resolution $0\to \SGm_2^\vee\to \SGm_1^\vee\to \SGm^\vee\to 0$ by the snake lemma. Put $\cG=\cG(\SGm)$ and $\Gamma_i=\cG(\SGm_i)$. By the snake lemma, we have an epimorphism $w: \Gamma_1[p^n]\to \cG$ whose dual map $w^\vee$ fits into the commutative diagram with exact rows
\[
\xymatrix{
0 \ar[r] & \cG(\SGm^\vee) \ar[r]\ar@{.>}[d]& \cG(\SGm_1^\vee) \ar[r]\ar[d]_{\wr} & \cG(\SGm_2^\vee) \ar[r]\ar[d]_{\wr} & 0\\
0 \ar[r] & \cG^\vee \ar[r]^{w^\vee}& \Gamma_1^\vee \ar[r] & \Gamma_2^\vee \ar[r]& 0.
}
\]
Thus we get an isomorphism $\cG(\SGm^\vee)\to \cG^\vee$, which is independent of the choice of a resolution (\cite[Lemma 2.3.4]{Ki_Fcrys}). On the other hand, we can check the commutativity of the diagram of the Cartier pairings
\[
\xymatrix{
T_p\Gamma_1 \ar[r]\ar[d] & \Hom(T_p(\Gamma_1^\vee), \bZ_p(1))\ar[d]\\
\cG(\okbar)\ar[r] &\Hom(\cG^\vee(\okbar),\bZ/p^n\bZ(1))
}
\]
and of a similar diagram for $\TSG(\SGm)$. Hence we can prove the compatibility with the duality pairings by the functoriality of the connecting homomorphism of the snake lemma. 
\end{proof}



\subsection{Proof of the main theorem} 

Now we prove Theorem \ref{main}. By Theorem \ref{TF}, Theorem \ref{ramdualp} and Proposition \ref{CarGM}, the assertion for the upper ramification subgroups is reduced to showing that the isomorphism 
\[
\delta_\SGm: \cG(\SGm)^\vee(\okbar)\to \TSG(\SGm^\vee)=\cH(\SGm^\vee)(R)
\]
induces an isomorphism of the lower ramification subgroups
\[
(\cG(\SGm)^\vee)_i(\okbar) \to \cH(\SGm^\vee)_i(R)
\]
for any $i\in \bQ_{\geq 0}$. By the definition of the map $\delta_\SGm$, it is enough to show the assertion of Theorem \ref{main} for the lower ramification subgroups. Namely, for an object $\SGm$ of the category $\ModSGf$, we reduced ourselves to showing the natural map 
\[
\varepsilon_\SGm: \cG(\SGm)(\okbar)\to \TSG(\SGm)=\cH(\SGm)(R)
\]
induces an isomorphism of the $i$-th lower ramification subgroups for any $i$. For this, by Proposition \ref{BCK} and replacing $K_1$ by $K$, we may assume that $e$ is divisible by $p$ and the entries of a representing matrix of $\phi_\SGm$ is contained in the subring $k[[u^p]]$ of $\SG_1$. Note that, by Corollary \ref{lowramboundp} and Remark \ref{lowrambound}, the $i$-th lower ramification subgroups of both sides vanish for $i>e/(p-1)$. Thus we are reduced to showing the theorem below. 

\begin{thm}\label{keythm}
Let $\SGm$ be an object of $\ModSGf$. Suppose that $e$ is divisible by $p$ and the entries of a representing matrix of $\phi_\SGm$ is contained in the subring $k[[u^p]]$ of $\SG_1$. Consider the isomorphism of $k$-algebras $k[u]/(u^e)\to \okey/p\okey$ defined by $u\mapsto \pi$, by which we identify both sides. Then there exists an isomorphism of schemes
\[
\eta_\SGm: (\okey/p\okey)\times_{\okey}\cG(\SGm) \to (k[u]/(u^e))\times_{k[[u]]}\cH(\SGm)
\]
which preserves the zero section and makes the following diagram commutative for any non-negative rational number $i\leq e$:
\[
\xymatrix{
\cG(\SGm)(\okbar)\ar[r]^{\varepsilon_\SGm}\ar[d] & \cH(\SGm)(R)\ar[d] \\
\cG(\SGm)(\okbar/{m}_{\Kbar}^{\geq i}) \ar[r]_-{\eta_\SGm} &\cH(\SGm)(R/{m}_{R}^{\geq i}).
}
\]
Here the bottom arrow is induced by the isomorphism $\prjt_0: R/{m}_{R}^{\geq i}\to \okbar/{m}_{\Kbar}^{\geq i}$ lying over the isomorphism $k[u]/(u^e)\to \okey/p\okey$.
\end{thm}
\begin{proof}
Let $m_1,\ldots,m_d$ be a basis of $\SGm$ such that we can write as 
\[
\phi_\SGm(m_1,\ldots, m_d)=(m_1,\ldots,m_d)A
\]
for some $A\in M_d(k[[u^p]])$. We can take matrices $P,Q\in GL_d(k[[u^p]])$ such that 
\[
P A Q=\diag(u^{e-r_1},\ldots,u^{e-r_d})
\]
for some non-negative integers $r_i$ divisible by $p$ with $r_i\leq e$. Here $\diag(a_1,\ldots,a_d)$ denotes the diagonal matrix whose $(i,i)$-th entry is $a_i$. Set a basis $n_1,\ldots, n_d$ of $\SGm$ to be $(n_1,\ldots,n_d)=(m_1,\ldots, m_d)\phi^{-1}(Q)$. Then we have
\[
\phi_\SGm(n_1,\ldots,n_d)=(n_1,\ldots,n_d)\phi^{-1}(Q)^{-1}P^{-1}\diag(u^{e-r_1},\ldots,u^{e-r_d}).
\]
Thus the object $\cM=\cM_\SG(\SGm)$ is described as 
\begin{align*}
&\Fil^1\cM=\Span_S(u^{r_1}\otimes n_1,\ldots,u^{r_d}\otimes n_d)+(\Fil^pS)\cM, \\
&\phi_1(u^{r_1}\otimes n_1,\ldots,u^{r_d}\otimes n_d)=(1\otimes n_1,\ldots,1\otimes n_d)G,
\end{align*}
where $G=cQ^{-1}\phi(P)^{-1}\in GL_d(S_1)$ with $c=\phi_1(E(u))$ as before. Take lifts $\hat{P}$ and $\hat{Q}$ of the matrices $P$ and $Q$ in $GL_d(W[[u^p]])$, respectively. Put $\hat{G}=\phi(-F(u))\hat{Q}^{-1}\phi(\hat{P})^{-1}\in GL_d(W[[u^p]])$. Since we have the equality $c=\phi(-F(u))$ in the ring $\tilde{S}_1$, the images of the matrices $G$ and $\hat{G}$ in $GL_d(\tilde{S}_1)$ coincide with each other. We write this image as $\tilde{G}$, which is contained in the subgroup $GL_d(k[u^p]/(u^{ep}))$.

Note that we have a commutative diagram of $W$-algebras
\[
\xymatrix{
W[[u]] \ar[rr]^{u\mapsto \pi} \ar[d] & &\cO_{K} \ar[d] \\
k[[u]] \ar[r] & k[u]/(u^{e})\ar[r]_-{\sim} & \okey/p\okey.
}
\]
Consider the composite map 
\[
W[[u]]\to \cO_{K}/p\cO_{K}\to \cO_{K_1}/p\cO_{K_1}\overset{\sim}{\to} k[u]/(u^{ep})=\tilde{S}_1,
\]
where the last arrow is the $\phi$-semilinear isomorphism defined by $\pi_1\mapsto u$. 
Then the image of the matrix $-F(u)\phi^{-1}(\hat{Q})^{-1}\hat{P}^{-1}$ by this composite map coincides with $\tilde{G}$. Let $a_{i,j}(u)\in W[[u]]$ be the $(i,j)$-th entry of this matrix. From the explicit description of the affine algebra $R_{\cM}$ of $\cG(\SGm)=\Gr^\mathrm{B}(\cM)$ recalled in Subsection \ref{RecallModGr}, we see that $R_{\cM}$ is defined by the system of equations over $\okey$
\[
X_i^p+\frac{\pi^{e-r_i}}{F(\pi)}(\sum_{j=1}^{d}a_{j,i}(\pi)X_j)\ \ (i=1,\ldots,d),
\]
where $a_{i,j}(\pi)$ denotes the image of $a_{i,j}(u)$ by the map $W[[u]]\to\okey$ defined as in the above diagram. On the other hand, the defining equations of $\cH(\SGm)$ over $k[[u]]$ are
\[
X_i^p+\frac{u^{e-r_i}}{F(u)}(\sum_{j=1}^{d}\bar{a}_{j,i}(u)X_j)\ \ (i=1,\ldots,d),
\]
where $\bar{a}_{i,j}(u)$ denotes the image of $a_{i,j}(u)$ by the natural map $W[[u]]\to k[[u]]$, and the zero section of $\cH(\SGm)$ is by definition $X_1=\cdots=X_d=0$. This implies that there exists an isomorphism
\[
\eta_\SGm: (\okey/p\okey)\times_{\okey}\cG(\SGm)\to (k[u]/(u^{e})) \times_{k[[u]]}\cH(\SGm)
\]
of schemes over the isomorphism $k[u]/(u^{e}) \simeq \okey/p\okey$ defined by $X_i\mapsto X_i$. Thus, for $i\leq e$, we get a bijection
\[
\eta_{\SGm}: \cG(\SGm)(\okbar/{m}_{\Kbar}^{\geq i}) \to \cH(\SGm)(R/{m}_{R}^{\geq i})
\]
satisfying $\eta_{\SGm}(0)=0$.

To prove the compatibility of $\varepsilon_\SGm$ and $\eta_{\SGm}$, let us consider the diagram
\[
\xymatrix{
\okbar \ar[r]\ar[d] & (\tokbar)^\PD & R^\PD \ar[l]_-{\prjt_1}^-{\sim} & R \ar[l]_-{\phi}\ar[d] \\
\tokbar & & & R/{m}_{R}^{\geq e}.\ar[lll]_{\prjt_0}^{\sim}
}
\]
Let $x=(x_1,\ldots,x_d)$ be an element of $\Spec(R_{\cM})(\okbar)$ and $z=(n_i\mapsto z_i)$ be the corresponding element of $\TSG(\SGm)=\Hom_{\SG,\phi}(\SGm,R)$ via the composite
\[
\Spec(R_{\cM})(\okbar)\simeq \cG(\SGm)(\okbar)\overset{\varepsilon_\SGm}{\to}T^{*}_{\SG}(\SGm).
\] 
Let $y_i\in R$ be the element such that $\prjt_1(y_i)$ coincides with the image $\bar{x}_i$ of $x_i$ in $\tokbar$. Then, in the ring $R^\PD$,  we have $\phi(z_i)-y_i\in \Fil^pR^\PD$. Put $y_i=(y_{i,0},y_{i,1},\ldots)$ and $z_i=(z_{i,0},z_{i,1},\ldots)$ with $y_{i,j},z_{i,j}\in\tokbar$. Since the natural map $R\to R^\PD$ induces an isomorphism
\[
R/{m}_{R}^{\geq ep}\to R^\PD/\Fil^pR^\PD
\]
and the kernel of the map $\prjt_1:R\to \tokbar$ coincides with the ideal ${m}_{R}^{\geq ep}$, we have $y_{i,1}=z_{i,1}^p=z_{i,0}$. This implies $\bar{x}_i=z_{i,0}$ and the compatibility of $\varepsilon_\SGm$ and $\eta_{\SGm}$ as in the theorem follows. Hence we conclude the proof of Theorem \ref{main}.
\end{proof}

Note that we have also shown the following corollary.

\begin{cor}
Let $\SGm$ be an object of the category $\ModSGf$. Consider the $k$-algebra $k[[v]]$ as a $k[[u]]$-algebra by the map $u\mapsto v^p$. By the $k$-algebra isomorphism $k[[v]]/(v^{ep})\to \okp/p\okp$ defined by $v \mapsto \pi_1$, we identify both sides. Then we have an isomorphism
\[
(\okp/p\okp)\times_{\okey}\cG(\SGm)\to (k[[v]]/(v^{ep}))\times_{k[[u]]}\cH(\SGm)
\]
of schemes over $k[[v]]/(v^{ep})\simeq \okp/p\okp$ preserving the zero section.
\end{cor}

\begin{rmk}
Let $\SGm$ be an object of the category $\ModSGr$. For $j\in \bQ_{>0}$ ({\it resp.} $i\in \bQ_{\geq 0}$), let $\SGm^j$ ({\it resp.} $\SGm_i$) be the object which corresponds via the anti-equivalence $\cH(-)$ to the closed subgroup scheme $\cH(\SGm)^j$ ({\it resp.} $\cH(\SGm)_i$) of $\cH(\SGm)$. These objects define cofiltrations $\{\SGm^j\}_{j\in \bQ_{>0}}$ and $\{\SGm_i\}_{i\in \bQ_{\geq 0}}$ of $\SGm$ in the category $\ModSGr$. Note that, for a finite flat group scheme over a discrete valuation ring, its finite flat closed subgroup scheme is determined by the generic fiber. Therefore, for $r=1$, Theorem \ref{main} and \cite[Theorem 3.4.3]{Br_AZ} imply that the quotient $\SGm\to \SGm^j$ ({\it resp.} $\SGm\to \SGm_i$) also corresponds via the anti-equivalence $\cG(-)$ to the closed subgroup scheme $\cG(\SGm)^j$ ({\it resp.} $\cG(\SGm)_i$) of $\cG(\SGm)$. They can be considered as ``upper and lower ramification cofiltrations'' of the Kisin module $\SGm$.
\end{rmk}

\begin{rmk}
The way we have proved Theorem \ref{main} is based on switching from the upper to the lower ramification subgroups via duality. The author wonders if we can prove the theorem in an ``upper" way, namely by constructing a natural isomorphism between the sets of geometric connected components of tubular neighborhoods of $\cG(\SGm)$ and $\cH(\SGm)$ using the similarity of their affine algebras, even though they are in different characteristics.
\end{rmk}

\noindent
{\bf Acknowledgments.} The author would like to thank Victor Abrashkin, Toshiro Hiranouchi, Wansu Kim and Yuichiro Taguchi for stimulating discussions, Xavier Caruso for kindly answering his questions on the paper \cite{Ca_D} and the anonymous referee for helpful suggestions to improve this paper. This work was supported by Grant-in-Aid for Young Scientists B-21740023.



\end{document}